\documentclass[11pt,a4paper,headinclude,footinclude]{amsart}                 
\usepackage{pifont}
\usepackage{fancyhdr,amsmath,amsthm,amscd,amssymb,amsfonts}
\usepackage{cite}
\usepackage{indentfirst, latexsym, bm, amsmath, eufrak, amsthm}
\usepackage[colorlinks,
linkcolor=blue,citecolor=red]{hyperref}
\usepackage{upgreek}
\usepackage{bm}
\usepackage{url}
\usepackage{lineno}
\usepackage{lipsum} 

\numberwithin{equation}{section}
\newtheorem{theo}{Theorem}[section]
\newtheorem*{theo*}{Theorem}

\newtheorem{lem}{Lemma}[section]
\newtheorem*{lem*}{Lemma}
\newtheorem{cor}{Corollary}[section]

\newtheorem{prop}{Proposition}[section]
\newtheorem{definition}{Definition}[section]
\newtheorem{example}{Example}[section]

\def\<{\langle}
\def\>{\rangle}

\begin{document}
	\title{An energy gap phenomenon for the Whitney sphere}
	
	
	\author{Liuyang Zhang}
	\address{Mathematisches Institut, Albert-Ludwigs-Universität Freiburg, Freiburg im Breisgau, 79104, Deutschland.}
	\email{davidmathematica@gmail.com}
	\date{}
	\maketitle
	
\begin{abstract}
	In this paper, we study Lagrangian surfaces satisfying $\nabla^*T=0$ , where $T=-2\nabla^*(\check{A}\lrcorner\omega)$ and $\check{A}$ is the Lagrangian trace-free second fundamental form. We obtain a gap lemma for such a Lagrangian surface.
\end{abstract}

	\section{Introduction}
	
	Gap phenomena form an interesting topic in differential geometry, with many related results to be found. Sacks-Uhlenbeck's well-known energy gap lemma for harmonic maps (see \cite{S81})  is such an example. The following result by Kuwert-Sch\"{a}tzle  for the Willmore surfaces immersed in Euclidean space is one of our motivations:
	

		\begin{theo*}[Gap lemma for Willmore surfaces, \text{\cite[Th.~1.1]{BR} or \cite[Th.~2.7]{KS01}}]
			Let $f:\Sigma\to\mathbb{R}^n$ be a properly immersed (compact or non-compact) Willmore surface, and let $\Sigma_\varrho (0):=f^{-1}(B_\varrho (0))$. Then there exists $\epsilon_0(n)>0$ such that if
			\begin{align*}
			&\liminf_{\rho\to\infty}\frac{1}{\varrho^4}\int_{\Sigma_\varrho (0)}|A|^2d\mu = 0\\
			&\quad and \quad\int_{\Sigma}|\AA{}|^2 d\mu<\epsilon_0(n),
			\end{align*}
			where \AA{} is the trace-free second fundamental form of $f(\Sigma)$,
		    $f$ is an embedded plane or sphere.
		\end{theo*}

The small energy condition above is natural in the variational sense since it can be geometrically interpreted as how different of an immersion from being the simplest geometric models such as planes and standard spheres. Our personal interest is on Lagrangian submanifolds which often play an important role in symplectic geometry where objects often have a natural presentation as Lagrangian manifolds. And they also 
arise in many problems of mechanics and physics. Models such as the Lagrangian planes, the Clifford torus and Whitney spheres are the simplest objects to study in Lagrangian geometry. Followed by Kuwert-Sch\"{a}tzle's idea, Luo-Wang proved a similar result under Lagrangian settings:

\begin{theo*}[Gap lemma for HW surfaces, \text{\cite[Th.~4.3]{LW15}}]
	Let $f:\Sigma\to\mathbb{C}^2$ be a properly immersed HW surface, then there exists $\epsilon_0(n)>0$ such that if the norm of the second fundamental form $\|A\|_{L^2}<\epsilon_0(n)$, it must be a Lagrangian plane.
\end{theo*}

According to their paper, it still remains open if there is a similar gap phenomenon for the Whitney sphere in the class of HW surfaces (see \cite{LW15} for the definition). To reformulate, we introduce a (0,2)-tensor $T:=\nabla(H\lrcorner\omega)-\frac{1}{2}\operatorname{div}JH\cdot g$ and consider the equation $\nabla^*T=0$ in this paper. From a geometric point of view, the tensor $T$ measures the deviation of the mean curvature vector field from a conformal field and one can easily check that Whitney spheres satisfy the equation. Therefore it is natural to ask if the Whitney sphere is unique under some small energy conditions. Instead of considering an energy condition on \AA{}, we introduce a Lagrangian trace-free second fundamental form $\check{A}$ (see definition in section 2) for its close relationship with Whitney spheres. The following is our main result:
\begin{theo}\label{Theorem 1.1}
	Assume $f:\Sigma\to\mathbb{C}^2$ is a properly immersed Lagrangian surface (compact or non-compact) such that $\nabla^*T=0$, given $\gamma\in C^1_c(\Sigma)$  a positive function that satisfies
	$|\nabla\gamma|\leq\frac{C_0}{R}$ for any $R>0$, there exists a constant $\epsilon_0>0$ such that if
	\begin{equation*}
	\int_{\{\gamma>0\}}|\check{A}|^2 d\mu\leq\epsilon_0,
	\end{equation*}
 we have
	\begin{equation*}
	\int _{\Sigma}(|\nabla\check{A}|^2+|H|^2|\check{A}|^2)\gamma^2d\mu\leq \frac{C}{R^2}\int_{\{\gamma>0\}}|A|^2 d\mu.
	\end{equation*}
\end{theo}
Combining the previous result with a classification theorem from Lagrangian geometry (see \cite{CU93} or \cite{CU01}), we obtain a gap theorem which answers the above question:
\begin{theo}\label{Theorem 1.2}
	Assume $f:\Sigma\to\mathbb{C}^2$ is a properly immersed Lagrangian surface (compact or non-compact) that satisfies $\nabla^*T=0$, and let ${\Sigma_{\varrho}(0)}=f^{-1}(B_{\varrho}(0))$, then there exists $\epsilon_0>0$ such that if
	\begin{equation*}
	\int_{\Sigma}|\check{A}|^2 d\mu\leq\epsilon_0\quad and\quad \liminf_{\varrho\to\infty}\frac{1}{\varrho^2}\int_{\Sigma_{\varrho}(0)}|A|^2 d\mu=0,
	\end{equation*}
 then $f$ is either a Lagrangian plane in $\mathbb{C}^2$ or a Whitney immersion.
\end{theo}
    Our method to prove these therorems is  establishing a Bochner type identity for the Lagrangian trace free curvature as Kuwert-Sch\"{a}tzle. But the difficulty here is that our condition on $\check{A}$ does not imply a control on $H$ as in their case. Therefore we write the Bochner identity in terms of the tensor $T$ so that we can make good use of its relationship with $\check{A}$.
    
    We organize this paper as follows: in Section 2 we introduce some elementary notions on Lagrangian submanifolds as well as the Willmore functional. Section 3 is devoted to a curvature estimate for Lagrangian surfaces which is essential for us to get the main gap theorem. In section 4, we will give a connection between our problem and the case of studying gap phenomena for the HW surfaces.

	\section{Preliminary on the Lagrangian geometry}
	
	In this section, let's recall some elementary notions in the Lagrangian geometry. Let $\mathbb{C}^2=\mathbb{R}^4$ be the 2-dimensional complex plane with the standard metric $ds^2=dx_i^2+dy_i^2$ (also denoted as $\langle\quad,\quad\rangle$) and $\omega=dx_i\wedge dy_i$ be the standard symplectic structure associated with it. Let $J$ be the standard complex structure of $\mathbb{C}^2$ such that $J^2=-id_{\mathbb{C}^2}$.  These  structures above  satisfy the relationship: $\langle V,W\rangle=\omega(V,JW)$ for any vectors $V$ and $W$. Here we order the coordinates as $(x_1,y_1,x_2,y_2)$. We denote the connection on $\mathbb{C}^2$ induced by the Euclidean metric as $D$. Now for an immersion $f:\Sigma\to\mathbb{C}^2$, we define the second fundamental form $A:=(D^2 f)^\bot$, i.e. the normal part of the second order covariant derivative of $f$, the mean curvature $H=trA$, and the trace-free second fundamental form $\text{\AA{}}:=A-\frac{1}{2}g\otimes H$ as usual. 
	
	\begin{definition}
		Let $\Sigma$ be a surface in $\mathbb{C}^2$, with tangent and normal bundles, $T\Sigma$ and $N\Sigma$, respectively. Then $\Sigma$ is Lagrangian if and only if one of the following equivalent conditions holds:
		\begin{itemize}
			\item[(1)] $\omega$ restricted to $\Sigma$ is zero,
			\item[(2)] $JT\Sigma=N\Sigma$, where $J$ is the standard complex structure on $\mathbb{C}^2$,
		\end{itemize}
	\end{definition}
   We usually treat $f(\Sigma)$ and $\Sigma$ as the same if there is no confusion. Hence the first condition has an alternative version:

	\begin{definition}
		An immersion $f$ from a surface $\Sigma$ into $\mathbb{C}^2$ is called a Lagrangian immersion if $f^*\omega =0$.
	\end{definition}
    The second condition allows us to choose frames properly which will be helpful when we are doing geometric calculations.
	
	The Lagrangian subspaces or planes are the simplest Lagrangian surfaces in $\mathbb{C}^2$, and there are also some well-known but non-trivial examples such as the Clifford torus and Whitney spheres.
	\begin{example}[Lagrangian planes in $\mathbb{C}^2$]
		All 2 dimensional subspaces of $\mathbb{R}^4$ whose  restriction of symplectic form $\omega$ to this subspace is identically equal to zero is called Lagrangian planes. 
	\end{example}

	\begin{example}[Whitney immersions in $\mathbb{C}^2$]
		\begin{equation}
		\begin{split}
		\Phi: \qquad\qquad\mathbb{S}^2&\longrightarrow \mathbb{C}^2
		\\ (x_1,x_2,x_3) &\longmapsto\frac{r}{1+x^2_3}(x_1,x_1 x_3,x_2,x_2 x_3)+\overrightarrow{C}
		\end{split}
		\end{equation}
		is a family of Lagrangian immersion. Here we embed $\mathbb{S}^2$ into $\mathbb{R}^3$ with center at the origin to get its local coordinates at first. The image of $\Phi$ in $\mathbb{C}^2$ is called a Whitney sphere and denoted as $\mathbb{S}_W$, and the constants $r$ and $\overrightarrow{C}$ will be referred as the radius and the center respectively.
	\end{example}
	Topologically, it is well-known that there is no embedded sphere in $\mathbb{C}^2$ as a Lagrangian submanifold. Whitney spheres have possibly the simplest behaviour in this case because they only have one double point.
	
	Castro-Urbano \cite{CU93} (or Ros-Urbano \cite{RU98} for higher dimensional case) proved the following famous classification theorem:
	\begin{theo*}[\text{\cite[Th.~2]{RU98}}~]
		Let $\Psi: M\to\mathbb{C}^n$ be a Lagrangian immersion of an n-dimensional submanifold M, then
		\begin{equation*}
		A(v,w)=\frac{1}{4}\{\langle v,w\rangle H+\langle Jv, H\rangle Jw+\langle Jw, H\rangle Jv\}
		\end{equation*}
		holds for any vectors v and w tangent to M if and only if $\Psi(M)$ is either an open set of the Whitney sphere or is totally geodesic.
	\end{theo*}
This reminds us of a classical theorem in $\mathbb{R}^n$ which states that if $\text{\AA{}}=0$, the immersion is either a plane or a standard sphere. Catro-Urbano's result is much more complicated to prove than this classical theorem. Motivated by this, we can define a similar quantity for Lagrangian surfaces:
	\begin{equation}
	\check{A}(v,w):=A(v,w)-\frac{1}{4}\{\langle v,w\rangle H+\langle Jv, H\rangle Jw+\langle Jw, H\rangle Jv\} \label{check A}.
	\end{equation}
	 We call $\check{A}$ the Lagrangian trace-free second fundamental form and an immersion is Lagrangian umbilical if it satisfies $\check{A}=0$.

	\section{Estimates for Lagrangian surfaces with locally small $L^2$-norm of $\check{A}$}

	\subsection{Preparations}
	Let $\{e_i\}$ be a local orthonormal frame for $\Sigma$ and  we denote $h_i=\langle H,Je_i\rangle=H\lrcorner\omega(e_i)$ and $A_{ijk}=\langle A(e_i.e_j), Je_k \rangle=A\lrcorner\omega(e_i,e_j,e_k)$ in local orthonormal coordinates. We would like to point out that $A$ is fully symmetric:
	\begin{equation}
		A_{ijk}=\langle D_{e_i}e_j,Je_k\rangle=-\langle e_j,JD_{e_i}e_k\rangle=\langle Je_j,D_{e_i}e_k\rangle=A_{ikj},\label{fully symmetric of A}
	\end{equation}
    where $D$ is the connection of the ambient space $\mathbb{C}^2$. Hence it doesn't matter which two of its three indices are contracted. We denote $g$ as the induced metric on $\Sigma$ by $f$,  $\nabla$ as its connection and $d\mu$ as the induced area form. For any tensor fields $S\in\Gamma(\underbrace{T\Sigma\otimes\cdots\otimes T\Sigma}_{r}\otimes N\Sigma)$ on $\Sigma$, 
    we define its covirant derivative as $\nabla ^{\bot} S=(D_X S)^\bot$ and its adjoint covariant derivative as $\nabla ^{\bot *}S=- e_i\lrcorner\nabla^{\bot}_{e_i}S$.  We can verify that they have the relationship
    \begin{equation*}
    \int_{\Sigma}\langle\nabla^\bot S,T\rangle d\mu=\int_{\Sigma}\langle S,\nabla^{^\bot*} T\rangle d\mu
    \end{equation*}
    for any  tensor fields $S\in \Gamma(\underbrace{T\Sigma\otimes\cdots\otimes T\Sigma}_{r}\otimes N\Sigma)$ and $T\in\Gamma(\underbrace{T\Sigma\otimes\cdots\otimes T\Sigma}_{r+1}\otimes N\Sigma)$.  In the following, we will omit the superscript of $\nabla^\bot$ if there is no confusion and hence the rough Laplace operator on the normal bundle can be written as $\varDelta S=-\nabla^* \nabla S$. The fundamental curvature functions of submanifold geometry can be expressed as
    \begin{align}
    &K=\frac{1}{2}(|H|^2-|A|^2)=\frac{1}{4}|H|^2-\frac{1}{2}|\text{\AA{}}|^2,\label{Gauss equ.}\\
    &(\nabla_X A)(Y,Z)=(\nabla_Y A)(X,Z), \quad(\nabla_X H)(Y,Z)=(\nabla_Y H)(X,Z),\label{Codazzi equ.}\\
    &R^{\bot}(X,Y)\phi=A(e_i,X)\langle A(e_i,Y),\phi\rangle-A(e_i,Y)\langle A(e_i,X),\phi\rangle.\label{Ricci equ.}
    \end{align}
    for any tangential vector fields $X$, $Y$ and normal vector fields $\phi$.
    
The following  proposition links those geometric quantities together:
	\begin{prop}
		Assume $f:\Sigma\to\mathbb{C}^2$ is a immersed Lagrangian surface (compact or non-compact), then its Lagrangian second fundamental form $\check{A}$ satisfies
		\begin{align}
		|\check{A}|^2&=|A|^2-\frac{3}{4}|H|^2 \label{(2.1)},\\
		-2\nabla^*(\check{A}\lrcorner\omega)&=\nabla(H\lrcorner\omega)-\frac{1}{2}\operatorname{div}JH\cdot g.
		\end{align}
	\end{prop}
	
	\begin{proof}
		In local orthonormal coordinates, we denote $\check{A}_{ijk}=\langle \check{A}(e_i.e_j), Je_k \rangle=\check{A}\lrcorner\omega(e_i,e_j,e_k)$. We use either  property (\ref{fully symmetric of A}) or  definition (\ref{check A}) to see that $\check{A}_{ijk}$ is fully symmetric. Hence we have
		\begin{align*}
		|\check{A}|^2&=\check{A}_{ijk}\check{A}^{ijk}\\
		&=[A_{ijk}-\frac{1}{4}(\delta_{ij} h_k+\delta_{ik} h_j+\delta_{jk} h_i)]^2\\
		&=A_{ijk} A^{ijk}-\frac{3}{4}h_i h^i.
		\end{align*}
		For the second identity, we use Codazzi equation (\ref{Codazzi equ.}) to get
		\begin{align*}
		-\nabla^*(\check{A}\lrcorner\omega)(e_i,e_j)&=\check{A}_{ijl,l}\\
		&=A_{ijl,l}-\frac{1}{4}(\delta_{ij}\operatorname{div}JH+h_{i,j}+h_{j,i})\\
		&=A_{ill,j}-\frac{1}{4}(\delta_{ij}\operatorname{div}JH+2h_{i,j})\\
		&=h_{i,j}-\frac{1}{4}(\delta_{ij}\operatorname{div}JH+2h_{i,j}).
		\end{align*}
	\end{proof}
    
    \begin{definition}
    	We define a (0,2)-tensor
    	\begin{equation*}
    		T:=-2\nabla^*(\check{A}\lrcorner\omega)=\nabla(H\lrcorner\omega)-\frac{1}{2}\operatorname{div}JH\cdot g,
    	\end{equation*}
    	or in local orthonormal basis:
    	\begin{equation*}
    		T_{ij}:=-2\check{A}_{ijl,l}=h_{i,j}-\frac{1}{2}h_{l,l}g_{ij},
    	\end{equation*}
    	where we have used Einstein's summation convention.
    \end{definition}
    One can see that $T$ is symmetric and actually the trace-free part of $\nabla(H\lrcorner\omega)$. To proceed we need a Bochner type identity which allows us to work globally. To make it clearer, instead of presenting it with coordinate free form straightly, we will introduce its local form at first and switch to its global form later in subsection 3.2.

	\begin{prop}[Bochner identity]
		If $f:\Sigma\to\mathbb{C}^2$ is a properly immersed Lagrangian surface, in local coordinates we have
		\begin{multline}\label{Laplacian of check A}
		\check{A}_{ijk,mm}=3K\check{A}_{ijk}\\
		+\frac{T_{ij,k}-\frac{1}{2}\delta_{ij}T_{km,m}}{3}+\frac{T_{jk,i}-\frac{1}{2}\delta_{jk}T_{im,m}}{3}+\frac{T_{ik,j}-\frac{1}{2}\delta_{ik}T_{jm,m}}{3}. 
		\end{multline}
	\end{prop}
	
	\begin{proof}
		Writing $\check{A}$ under local coordinates:
		\begin{align*}
		\check{A}_{ijk,mm}&=\check{A}_{ijm,km}+\frac{1}{4}(\delta_{ij}h_{m,km}-\delta_{ij}h_{k,mm}+\delta_{im}h_{j,km}-\delta_{ik}h_{j,mm}\\&\quad+\delta_{jm}h_{i,km}-\delta_{jk}h_{i,mm}),
		\end{align*}
		we commute the second order derivative of $\check{A}$ by Ricci's identity:
		\begin{equation*}
		 \check{A}_{ijm,km}=\check{A}_{ijm,mk}+\check{A}_{ljm}R_{ikm}^l+\check{A}_{lim}R_{jkm}^l+\check{A}_{lij}R_{mkm}^l.
		\end{equation*}
		Hence
		\begin{align*}
		& R.H.S.=\check{A}_{ijm,mk}+\check{A}_{ljm}R_{ikm}^l+\check{A}_{lim}R_{jkm}^l+\check{A}_{lij}R_{mkm}^l+\frac{1}{4}(h_{j,ki}+h_{i,kj}\\&\quad-\delta_{ik}h_{j,mm}-\delta_{jk}h_{i,mm}).
		\end{align*}
		Now by the definition of $\check{A}$ again, we may commute the order of $j$ and  $m$:
		\begin{equation*}
			\check{A}_{ijm,mk}=\check{A}_{imm,jk}+\frac{1}{4}(\delta_{im}h_{m,jk}-\delta_{ij}h_{m,mk}+\delta_{im}h_{m,jk}-\delta_{im}h_{j,mk}+\delta_{mm}h_{i,jk}-\delta_{jm}h_{i,mk}).
		\end{equation*}
		So
		\begin{align*}
		&R.H.S.=\check{A}_{imm,jk}+\check{A}_{ljm}R_{ikm}^l+\check{A}_{lim}R_{jkm}^l+\check{A}_{lij}R_{mkm}^l+\frac{1}{4}(h_{j,ki}+h_{i,kj}\\&\quad-\delta_{ik}h_{j,mm}-\delta_{jk}h_{i,mm})+\frac{1}{4}(\delta_{im}h_{m,jk}-\delta_{ij}h_{m,mk}+\delta_{im}h_{m,jk}\\&\quad-\delta_{im}h_{j,mk}+\delta_{mm}h_{i,jk}-\delta_{jm}h_{i,mk}),
		\end{align*}
		where $\check{A}_{imm,jk}=0$ because $\check{A}$ is trace-free.
		
		For surfaces, we have $R_{ijkl}=K(g_{ik} g_{jl}-g_{il} g_{jk})$ by definition. It holds
		\begin{align*}
		&R.H.S.=\check{A}_{ljm}K(\delta_{lk}\delta_{im}-\delta_{lm}\delta_{ik})+\check{A}_{lim}K(\delta_{lk}\delta_{jm}-\delta_{lm}\delta_{jk})+\check{A}_{lij}K(\delta_{lk}\delta_{mm}\\
		&\quad-\delta_{lm}\delta_{km})+\frac{1}{4}(h_{j,ki}+h_{i,kj}-\delta_{ik}h_{j,mm}-\delta_{jk}h_{i,mm})+\frac{1}{4}(\delta_{im}h_{m,jk}\\
		&\quad-\delta_{ij}h_{m,mk}+\delta_{im}h_{m,jk}-\delta_{im}h_{j,mk}+\delta_{mm}h_{i,jk}-\delta_{jm}h_{i,mk})\\
		&=3K\check{A}_{ijk}+\frac{1}{4}(2h_{i,jk}+h_{j,ki}+h_{i,kj})-\frac{1}{4}(\delta_{ik}h_{j,mm}+\delta_{jk}h_{i,mm}+\delta_{ij}h_{m,mk})\\
		&=3K\check{A}_{ijk}+\frac{2h_{i,jk}+h_{j,ki}+h_{i,kj}-\delta_{ik}h_{j,mm}-\delta_{jk}h_{i,mm}-\delta_{ij}h_{m,mk}}{4}.
		\end{align*}
		Substituting terms like $h_{i,jk}$ with $(T_{ij,k}+\frac{1}{2}\delta_{ij}h_{l,lk})$ above, we get
		\begin{align*}
		&R.H.S.=3K\check{A}_{ijk}+\frac{2T_{ij,k}+T_{jk,i}+T_{ik,j}}{4}-\frac{\delta_{jk}h_i+\delta_{ik}h_j}{8}K-\frac{\delta_{ik}h_{j,mm}+\delta_{jk}h_{i,mm}}{8}\\
		&=3K\check{A}_{ijk}+\frac{2T_{ij,k}+T_{jk,i}+T_{ik,j}-\delta_{ik}T_{jm,m}-\delta_{jk}T_{im,m}}{4},
		\end{align*}
		Since $\check{A}$ is fully symmetric, we can apply the method of symmetrization to get (\ref{Laplacian of check A}) as desired.
	\end{proof}
	
	\subsection{Curvature estimates} 
	
	The methods we use in this  part are similar to Kuwert - Sch\"{a}tzle's work in \cite{KS01}.
	
	\begin{lem}
		Assume $f:\Sigma\to\mathbb{C}^2$ is a properly immersed Lagrangian surface (compact or non-compact), and let  $\gamma$ be a cut-off function with $\|\nabla\gamma\|_{L^{\infty}}=\Gamma$, then we have:
		\begin{multline}
		\int_{\Sigma} (|\nabla\check{A}|^2+|H|^2|\check{A}|^2)\gamma^2d\mu\\
		\leq C\int_{\Sigma}\langle\nabla^*T,H\rangle\gamma^2 d\mu+ C\int_{\Sigma}|\check{A}|^4\gamma^2 d\mu+C\Gamma^2\int_{\{\gamma>0\}}|A|^2 d\mu.
		\end{multline}\label{estimates for nabla check A}
		where $C$ is a positive constant independent of $f$.
	\end{lem}
	\begin{proof}
		Multiplying (\ref{Laplacian of check A}) by $\check{A}$ we get
		\begin{multline}\label{Laplacian of check A}
		\check{A}_{ijk,mm}\check{A}^{ijk}=3K\check{A}_{ijk}\check{A}^{ijk}\\
		+(\frac{T_{ij,k}-\frac{1}{2}\delta_{ij}T_{km,m}}{3}+\frac{T_{jk,i}-\frac{1}{2}\delta_{jk}T_{im,m}}{3}+\frac{T_{ik,j}-\frac{1}{2}\delta_{ik}T_{jm,m}}{3})\check{A}^{ijk}.
		\end{multline}
		Since $\check{A}$ is trace-free, terms like $\delta_{ij}T_{km,m}\check{A}^{ijk}$ will vanish. Now  (\ref{Laplacian of check A}) becomes
		\begin{equation}
			\langle \varDelta\check{A},\check{A}\rangle=\langle\nabla T,\check{A} \rangle+3K|\check{A}|^2.
		\end{equation}
		Using the Gau\ss{} equation (\ref{Gauss equ.}) and proposition 3.1, we have $K=\frac{|H|^2}{8}-\frac{|\check{A}|^2}{2}$. Multiplying (3.10) by $\gamma^2$ and integrating it on the surface, we have
		\begin{equation}
		\int_{\Sigma}\langle \varDelta\check{A},\gamma^2\check{A}\rangle d\mu=\int_{\Sigma}\langle\nabla T,\gamma^2\check{A} \rangle d\mu+\frac{3}{8}\int_{\Sigma}|H|^2|\check{A}|^2\gamma^2 d\mu-\frac{3}{2}\int_{\Sigma} |\check{A}|^4\gamma^2 d\mu.
		\end{equation}
		As for the L.H.S., we obtain
		\begin{align*}
		L.H.S.&=-\int_{\Sigma}\langle\nabla^*\nabla\check{A},\gamma^2\check{A}\rangle d\mu\\
		&=-\int_{\Sigma}|\nabla\check{A}|^2\gamma^2d\mu-2\int_{\Sigma}\langle\nabla\check{A},\gamma\nabla\gamma\otimes\check{A}\rangle d\mu.
		\end{align*}
		Now for the first term of R.H.S., we use the definition of $T$ and integrate it by parts
		\begin{align*}
		\int_{\Sigma}\langle \nabla T,\gamma^2\check{A}\rangle d\mu&=\int_{\Sigma}\langle T,\gamma^2\nabla^*\check{A}\rangle d\mu-2\int_{\Sigma}\langle T\otimes\nabla\gamma,\gamma\check{A}\rangle d\mu\\
		&=-\frac{1}{2}\int_{\Sigma}\langle T,T\rangle\gamma^2 d\mu-2\int_{\Sigma}\langle T\otimes\nabla\gamma,\gamma\check{A}\rangle d\mu\\
		&=-\frac{1}{2}\int_{\Sigma}\langle T,\nabla H-\frac{1}{2}\operatorname{div}JH\cdot g\rangle\gamma^2 d\mu-2\int_{\Sigma}\langle T\otimes\nabla\gamma,\gamma\check{A}\rangle d\mu\\
		&=-\frac{1}{2}\int_{\Sigma}\langle T,\nabla H\rangle\gamma^2 d\mu-2\int_{\Sigma}\langle T\otimes\nabla\gamma,\gamma\check{A}\rangle d\mu\\
		&=-\frac{1}{2}\int_{\Sigma}\langle\nabla^*T,H\rangle\gamma^2 d\mu+\int_{\Sigma}\langle T,H\otimes\nabla\gamma\rangle\gamma d\mu-2\int_{\Sigma}\langle T\otimes\nabla\gamma,\gamma\check{A}\rangle d\mu.
		\end{align*}
		
		Hence with the help of $|T|\leq c|\nabla\check{A}|$ and $|\nabla\gamma|\leq C_0\Gamma^2$ we have
		\begin{align*}
		\int_{\Sigma}(&|\nabla\check{A}|^2+\frac{3}{8}|H|^2|\check{A}|^2)\gamma^2 d\mu\\
		&=\frac{1}{2}\int_{\Sigma}\langle\nabla^*T,H\rangle\gamma^2 d\mu+\frac{3}{2}\int_{\Sigma}|\check{A}|^4\gamma^2 d\mu-2\int_{\Sigma}\langle\nabla\check{A},\gamma\nabla\gamma\otimes\check{A}\rangle d\mu\\
		&\quad-\int_{\Sigma}\langle T,H\otimes\nabla\gamma\rangle\gamma d\mu+2\int_{\Sigma}\langle T\otimes\nabla\gamma,\gamma\check{A}\rangle d\mu\\
		&\leq\frac{1}{2}\int_{\Sigma}\langle\nabla^*T,H\rangle\gamma^2 d\mu+\frac{3}{2}\int_{\Sigma}|\check{A}|^4\gamma^2 d\mu+C_0\Gamma^2\int_{\{\gamma>0\}}|\check{A}|^2 d\mu+\frac{1}{2}\int_{\Sigma}|\nabla\check{A}|^2\gamma^2 d\mu\\
		&\quad+C_0\Gamma^2\int_{\{\gamma>0\}}|H|^2 d\mu\\
		&\leq\frac{1}{2}\int_{\Sigma}\langle\nabla^*T,H\rangle\gamma^2 d\mu+\frac{3}{2}\int_{\Sigma}|\check{A}|^4\gamma^2 d\mu+\frac{1}{2}\int_{\Sigma}|\nabla\check{A}|^2\gamma^2 d+C\Gamma^2\int_{\{\gamma>0\}}|A|^2 d\mu.
		\end{align*}
	\end{proof}
	
	We now need the general Sobolev inequality of Michael-Simon to absorb $\int_{\Sigma}|\check{A}|^4\gamma^2 d\mu$.
	
	\begin{theo*}[Sobolev inequality with $m=2$, \text{\cite[Th.~2.1]{MS73}}]
		Let $f:\Sigma\to\mathbb{C}^2$ be an immersion and $v$ be a non-negative $C^1_c(U)$ function on $\Sigma$, where $U\subseteq \mathbb{C}^2$ is a domain contains $f(\Sigma)$. Then
		\begin{equation}
		\int_\Sigma v^2d\mu\leq c\Bigg(\int_\Sigma|\nabla v|d\mu\Bigg)^2+c\Bigg(\int_\Sigma v|H|d\mu\Bigg)^2\label{M-S Sobolev},
		\end{equation}
		where $H$ is the mean curvature vector and c is a constant independent of $f$ .
	\end{theo*}
	
	\begin{lem}
		Under the same assumption as in Lemma 3.1,
		\begin{equation}
		\int_{\Sigma}|\check{A}|^4\gamma^2 d\mu\leq C\int_{\{\gamma>0\}}|\check{A}|^2 d\mu\int_{\Sigma}(|\nabla\check{A}|^2+|H|^2|\check{A}|^2)\gamma^2 d\mu+C\Gamma^2\Bigg(\int_{\{\gamma>0\}}|\check{A}|^2d\mu\Bigg)^2
		\end{equation}
	holds.
	\end{lem}
	\begin{proof}
		Substituting $v=|\check{A}|^2\gamma$ in (\ref{M-S Sobolev}), we have
		\begin{align*}
		\int_{\Sigma}&|\check{A}|^4\gamma^2 d\mu\\
		&\leq c\Bigg(\int_{\Sigma}|\check{A}||\nabla\check{A}|\gamma d\mu\Bigg)^2+c\Bigg(\int_{\Sigma}|\check{A}|^2|\nabla\gamma| d\mu\Bigg)^2+c\Bigg(\int_{\Sigma} |\check{A}|^2|H|\gamma d\mu\Bigg)^2\\
		&\leq c'\int_{\{\gamma>0\}}|\check{A}|^2 d\mu\int_{\Sigma}(|\nabla\check{A}|^2+|H|^2|\check{A}|^2)\gamma^2 d\mu+c\Gamma^2\Bigg(\int_{\{\gamma>0\}}|\check{A}|^2d\mu\Bigg)^2.
		\end{align*}
	\end{proof}

\begin{proof}[Proof of the Theorem \ref{Theorem 1.1}]
	By combining Lemma 3.1 and Lemma 3.2, one can straightforwardly get
	\begin{align*}
	\int _{\Sigma}(|\nabla\check{A}|^2+|H|^2|\check{A}|^2)\gamma^2d\mu&\leq C'\int_{\{\gamma>0\}}|\check{A}|^2 d\mu\int_{\Sigma}(|\nabla\check{A}|^2+|H|^2|\check{A}|^2)\gamma^2 d\mu\\
	&+C'\Gamma^2\int_{\{\gamma>0\}}|A|^2 d\mu+C'\Gamma^2\Bigg(\int_{\{\gamma>0\}}|\check{A}|^2d\mu\Bigg)^2.
	\end{align*}
	If we choose $\epsilon_0<min\{1,\frac{1}{2C'}\}$ and assume $\int_{\{\gamma>0\}}|\check{A}|^2 d\mu\leq\epsilon_0$, it holds 
	\begin{equation*}
		\int _{\Sigma}(|\nabla\check{A}|^2+|H|^2|\check{A}|^2)\gamma^2d\mu\leq\frac{C''\Gamma^2}{1-C'\int_{\{\gamma>0\}}|\check{A}|^2 d\mu}\int_{\{\gamma>0\}}|A|^2 d\mu\leq \tilde{C}\Gamma^2\int_{\{\gamma>0\}}|A|^2 d\mu,
	\end{equation*}
	let $\psi\in C^1(\mathbb{R})$ be a non-negative function which satisfies:
	
	\begin{equation*}
	\psi(t)=
	\begin{cases}
	1,& t\leq \frac{1}{2},\\
	0,& t\geq 1.
	\end{cases}
	\end{equation*}
	and $\gamma(p)=\psi(\frac{1}{R}|f(p)|)\in C^1_c(\Sigma)$ be the cut-off function for any $p\in\Sigma$ and radius $R>0$, then $\Gamma\leq\frac{C_0}{R}$  holds and we obtain
	\begin{equation*}
		\int _{\Sigma}(|\nabla\check{A}|^2+|H|^2|\check{A}|^2)\gamma^2d\mu\leq \frac{C}{R^2}\int_{\{\gamma>0\}}|A|^2 d\mu
	\end{equation*}
	as dersired.
\end{proof}

\begin{proof}[Proof of the Theorem \ref{Theorem 1.2}]
	 Letting $R\to\infty$ in (3.14), we have $\check{A}=0$ by the arbitrariness of cut off function $\gamma$. Then using the classification theorem for Lagrangian umblical surfaces (see \cite{CU93,RU98}), we see that $f$ is either a Lagrangian plane or a Whitney sphere.
\end{proof}

 \section{On the relationship with Willmore immersion}
 
For an immersed surface $f:\Sigma \to \mathbb{R}^n$, the Willmore functional is defined by
\begin{equation}
\mathcal{W}(f)=\int_{\Sigma}|\text{\AA{}}|^2 d\mu \label{Willmore Functional},
\end{equation}
where we denote $d\mu$ as the induced area element on $\Sigma$ by $f$.

The Euler-Lagrange operator of (\ref{Willmore Functional}) is
\begin{equation}
\operatorname{W}=\varDelta H+Q(\text{\AA{}})H, \label{E-L equ-1}
\end{equation}
where
\begin{equation}
Q(\text{\AA{}})\phi=\Sigma_{i,j=1}^{2}\text{\AA{}}(e_{i},e_{j})\langle\text{\AA{}}(e_i,e_j),\phi\rangle \label{definition of Q(A)H}
\end{equation}
under an orthonormal basis $\{e_{1},e_2{}\}$.
The following lemma reveals the relationship between $T$ and the Euler-Lagrange equation of the Willmore functional:
\begin{lem}
	If $f:\Sigma\to\mathbb{C}^2$ is a properly immersed Lagrangian surface, the Euler-Lagrange equation of the Willmore functional can be reformulated as:
	\begin{equation*}
	\operatorname{W}=\varDelta H+\frac{1}{8}|H|^2 H+Q(\check{A})H+\check{A}(JH,JH).
	\end{equation*}
	or in the dual form associated with the symplectic form $\omega$:
	\begin{equation}
	\operatorname{W}\lrcorner\omega=-2\nabla^* T+\frac{1}{2}|\check{A}|^2 H\lrcorner\omega+Q(\check{A})H+\check{A}\lrcorner\omega(JH,JH). \label{E-L equ.2}
	\end{equation}
\end{lem}
\begin{proof}
	By (\ref{E-L equ-1}) and (\ref{definition of Q(A)H}), we have
	\begin{align*}
	&Q(A)\\
	&=A(e_i,e_j)\langle A(e_i,e_j),H\rangle \\
	&=[\check{A}_{ijh}+\frac{1}{4}(\delta_{ij}h_h+\delta_{ih}h_j+\delta_{jh}h_i)][\check{A}_{ijs}+\frac{1}{4}(\delta_{ij}h_s+\delta_{is}h_j+\delta_{js}h_i)]\langle Je_s,H\rangle Je_h \\
	&=\check{A}_{ijs}\langle \check{A}_{ijh},H\rangle\langle Je_s,H\rangle Je_h+\check{A}_{ijh} h_i h_j Je_h+\frac{5}{8}|H|^2H\\
	&=Q(\check{A})H+\check{A}(JH,JH)+\frac{5}{8}|H|^2H.
	\end{align*}
	Using the Ricci identity (\ref{Ricci equ.}), the Gau\ss{} equation (\ref{Gauss equ.}) and the Codazzi equation (\ref{Codazzi equ.}), we calculate
	\begin{align*}
	\nabla^*T(e_i)&=-T_{ki,k}\\
	&=-h_{i,kk}+\frac{1}{2}\delta_{ki}h_{l,lk} \\
	&=-h_{i,kk}+\frac{1}{2}(h_{i,ll}+h_s R_{lli}^s) \\
	&=-\frac{1}{2}\varDelta(H\lrcorner\omega)(e_i)-\frac{K}{2}h_i\\
	&=-\frac{1}{2}\varDelta (H\lrcorner\omega)(e_i)-\frac{1}{2}(\frac{|H|^2}{8}-\frac{|\check{A}|^2}{2})H\lrcorner\omega(e_i),
	\end{align*}
	where $K$ is the Gaussian curvature of the surface. Substituting the above two formulas into (\ref{E-L equ-1}) we complete the proof.
\end{proof}

 The Willmore Lagrangian surface is a Lagrangian surface satisfying the Euler-Lagrange equation (\ref{E-L equ-1}) Now let's recall Urbano and Castro's classification theorem as follows:
 
\begin{theo*}[Uniqueness of Willmore Lagrangian spheres,\text{\cite[Cor.~1]{CU01}}]
	The Whitney sphere is the only Willmore Lagrangian surface of genus 0 in $\mathbb{C}^2$.
\end{theo*}

Then we obtain a gap-lemma for Willmore Lagrangian surfaces.

\begin{cor}[Gap theorem for Willmore Lagrangian surfaces]
	Let $f:\Sigma\to\mathbb{C}^2$ be a closed Lagrangian surface satisfies $\operatorname{W}=0$, then there exists a universal constant $\epsilon_0$ such that if
	\begin{equation*}
	\int_{\Sigma}|\check{A}|^2 d\mu\leq\epsilon_0,
	\end{equation*}
 $f$ is a Whitney immersion.
\end{cor}
\begin{proof}
	Using the Gau\ss{}-Bonnet formula and (\ref{(2.1)}), we see that if $f$ is Lagrangian immersion, we have
	\begin{equation*}
	\int_{\Sigma}|\check{A}|^2d\mu=\frac{1}{2}\int_{\Sigma}|\text{\AA{}}|^2 d\mu-2\pi\chi(\Sigma).
	\end{equation*}
	If $\int_{\Sigma}|\check{A}|^2d\mu$ is sufficiently small, $\chi(\Sigma)$ must be non-negative, which means that it is either a sphere or a torus. However the result of \cite{KS01} eliminates the torus case given the Willmore energy is sufficiently small.
\end{proof}

\section*{Acknowledgement}
	
	The author would like to thank professor Guofang Wang for his many inspiring discussions. The project is supported by  "Willmore Funktional und Langrangsche Fl\"{a}chen" in SPP 2026 of DFG.


\end{document}